\documentclass[12pt]{amsart}

\usepackage{amsmath,amsfonts,amssymb,mathabx,shuffle,latexsym}
\usepackage{enumitem}
\usepackage[usenames, dvipsnames]{xcolor}
\usepackage[backref=page]{hyperref}
\usepackage{ytableau}
\usepackage{tikz-cd}

\newtheorem{theorem}{Theorem}[section]
\newtheorem{lemma}[theorem]{Lemma}

\newtheorem{corollary}[theorem]{Corollary}

\theoremstyle{definition}
\newtheorem{definition}[theorem]{Definition}
\newtheorem{examplex}[theorem]{Example}
\newenvironment{example}
  {\pushQED{\qed}\examplex}
  {\popQED\endexamplex}

\theoremstyle{remark}
\newtheorem{remark}[theorem]{Remark}

\numberwithin{equation}{section}

\DeclareMathOperator{\codim}{codim}
\DeclareMathOperator{\ex}{ex}
\newcommand{\wt}{\ensuremath{\mathrm{wt}}}
\newcommand{\kwt}{\ensuremath{\mathrm{kwt}}}
\newcommand{\rev}{\ensuremath{\mathrm{rev}}}
\newcommand{\inv}{\ensuremath{\mathrm{inv}}}
\newcommand{\schub}{\ensuremath{\mathfrak{S}}}
\newcommand{\groth}{\ensuremath{\mathfrak{G}}}
\newcommand{\cgroth}{\ensuremath{\mathfrak{K}}}
\newcommand{\Fund}{\ensuremath{\mathfrak{F}}}
\newcommand{\glide}{\ensuremath{\mathcal{G}}}

\newcommand{\flatten}{\ensuremath{\mathtt{flat}}}
\newcommand{\reduct}{\ensuremath{\mathtt{reduct}}}

\newcommand{\PD}{\ensuremath{\mathrm{PD}}}
\newcommand{\QPD}{\ensuremath{\mathrm{QPD}}}

\newcommand{\destand}{\ensuremath{\mathrm{dst}}}

\newcommand{\gen}{\ensuremath{\mathrm{gen}}}

\newcommand{\SV}{\ensuremath{\mathrm{SV}}}
\newcommand{\QSV}{\ensuremath{\mathrm{QSV}}}
\newcommand{\QSym}{\ensuremath{\mathrm{QSym}}}
\newcommand{\Poly}{\ensuremath{\mathrm{Poly}}}

\newlength\cellsize 

\savebox2{%
\begin{picture}(30,30)
\put(0,0){\line(1,0){30}}
\put(0,0){\line(0,1){30}}
\put(30,0){\line(0,1){30}}
\put(0,30){\line(1,0){30}}
\end{picture}}

\newcommand\cellify[1]{\def\thearg{#1}\def\nothing{}%
\ifx\thearg\nothing\vrule width0pt height\cellsize depth0pt%
  \else\hbox to 0pt{\usebox2\hss}\fi%
  \vbox to 30\unitlength{\vss\hbox to 30\unitlength{\hss$#1$\hss}\vss}}

\newcommand\tableau[1]{\vtop{\let\\=\cr
\setlength\baselineskip{-12000pt}
\setlength\lineskiplimit{12000pt}
\setlength\lineskip{0pt}
\halign{&\cellify{##}\cr#1\crcr}}}

\newcommand\turn{
\begin{picture}(10,10)
\thicklines
\put(0,10){\oval(10,10)[br]}
\put(10,0){\oval(10,10)[tl]}
\end{picture}}

\newcommand\tail{
\begin{picture}(10,10)
\thicklines
\put(0,10){\oval(10,10)[br]}
\end{picture}}

\newcommand\cross{
\begin{picture}(10,10)
\thicklines
\put(5,0){\line(0,1){10}}
\put(0,5){\line(1,0){10}}
\end{picture}}

\newcommand\gridify[1]{\vbox to 10\unitlength{\vss\hbox to 10\unitlength{\hss$_{#1}$\hss}\vss}}

\newcommand\pipes[1]{\vtop{\let\\=\cr
\setlength\baselineskip{-10000pt}
\setlength\lineskiplimit{10000pt}
\setlength\lineskip{0pt}
\halign{&\gridify{##}\cr#1\crcr}}}

\usepackage[colorinlistoftodos]{todonotes}

\begin{document}


\title{Decompositions of Grothendieck Polynomials}  

\author[O. Pechenik]{Oliver Pechenik}
\address[OP]{Department of Mathematics, Rutgers University, Piscataway, NJ 08854}
\email{oliver.pechenik@rutgers.edu}

\author[D. Searles]{Dominic Searles}
\address[DS]{Department of Mathematics, University of Southern California, Los Angeles, CA 90089}
\email{dsearles@usc.edu}


\date{\today}



\begin{abstract}
We investigate the longstanding problem of finding a combinatorial rule for the Schubert structure constants in the $K$-theory of flag varieties (in type $A$). The Grothendieck polynomials of A.~Lascoux--M.-P.~Sch\"{u}tzenberger (1982) serve as polynomial representatives for $K$-theoretic Schubert classes; however no positive rule for their multiplication is known outside the Grassmannian case. We contribute a new basis for polynomials, give a positive combinatorial formula for the expansion of Grothendieck polynomials in these \emph{glide polynomials}, and provide a positive combinatorial Littlewood-Richardson rule for expanding a product of Grothendieck polynomials in the glide basis. 
Our techniques easily extend to the $\beta$-Grothendieck polynomials of S.~Fomin--A.~Kirillov (1994), representing classes in connective $K$-theory, and we state our results in this more general context. 

A specialization of the glide basis recovers the fundamental slide polynomials of S.~Assaf--D.~Searles (2016), which play an analogous role with respect to the Chow ring of flag varieties. Additionally, the stable limits of another specialization of glide polynomials are T.~Lam--P.~Pylyavskyy's (2007) basis of multi-fundamental quasisymmetric functions, $K$-theoretic analogues of I.~Gessel's (1984) fundamental quasisymmetric functions. Those glide polynomials that are themselves quasisymmetric are truncations of multi-fundamental quasisymmetric functions and form a basis of quasisymmetric polynomials.
\end{abstract}

\maketitle
\tableofcontents

%
\section{Introduction}
%
\label{sec:introduction}

Let $X={\sf Flags}({\mathbb C}^n)$ be the parameter space of
complete flags 
\[{\mathbb C}^0\subset F_1\subset F_2\subset
\cdots\subset F_{n-1}\subset {\mathbb C}^n.\] 
The space $X$ is a smooth projective complex variety and carries an action of ${\sf GL_n}(\mathbb{C})$ induced from the standard action of ${\sf GL_n}(\mathbb{C})$ on $\mathbb{C}^n$. There are then restricted actions by the {\bf Borel subgroup} ${\sf B}$ of invertible lower triangular matrices and the {\bf maximal torus} ${\sf T}$ of 
invertible diagonal matrices.
The ${\sf T}$-fixed points of $X$ are the flags 
$F_{\bullet}^{(w)}$ defined by $F_k^{(w)}=\langle  e_{w(1)}, e_{w(2)},\ldots,  e_{w(k)}\rangle$, where $e_i$ is the $i$th standard basis vector and $w \in S_n$ is a permutation. Hence the ${\sf T}$-fixed points are naturally indexed by the permutations $w$ in the symmetric group $S_n$. Each ${\sf B}$-orbit of $X$ contains a unique ${\sf T}$-fixed point, and the {\bf Schubert varieties} 
\[X_w := \overline{{\sf B} \cdot F_{\bullet}^{(w)}}\]
 give a cell decomposition of $X$.

Since the structure sheaf $\mathcal{O}_{X_w}$ of a Schubert variety has a resolution by locally free sheaves
\[
0 \to V_k \to V_{k-1} \to \dots \to V_0  \to \mathcal{O}_{X_w} \to 0,
\]
one may thereby define classes 
\[[\mathcal{O}_{X_w}] := \sum_{i=0}^k (-1)^i [V_i] \] in the Grothendieck ring $K(X)$ of algebraic vector bundles over $X$.
Indeed the set $\{ [\mathcal{O}_{X_w}] \}_{w \in S_n}$ of these $K$-theoretic Schubert classes is an additive basis for $K(X)$. Hence the product structure of $K(X)$ (given by tensor product of vector bundles) is encoded in the structure coefficients $C_{u,v}^w$ appearing in 
\[
[\mathcal{O}_{X_u}] \cdot [\mathcal{O}_{X_v}] = \sum_{w \in S_n} C_{u,v}^w  [\mathcal{O}_{X_w}].
\]
It was conjectured by A.~Buch \cite{Buch} and proved by M.~Brion \cite{Brion} that the signs of these coefficients are determined simply by the codimensions of the Schubert varieties in $X$. More precisely, $(-1)^{\ell(w) - \ell(u) - \ell(v)} C_{u,v}^w \geq 0$, where $\ell(w) = \codim_X(X_w)$ (or equivalently the Coxeter length of $w$). 

Since the numbers $(-1)^{\ell(w) - \ell(u) - \ell(v)} C_{u,v}^w$ are nonnegative, one might hope for a combinatorial rule expressing $|C_{u,v}^w|$ as the cardinality of some explicit set of combinatorial objects. Giving such a rule remains a major, long-standing problem in algebraic combinatorics.

We address (but do not solve) this problem. Most important of the available combinatorial tools are the \emph{Grothendieck polynomials} $\groth_w$ (introduced by A.~Lascoux--M.-P.~Sch\"utzenberger \cite{Lascoux.Schutzenberger}), which are polynomial representatives for the $K$-theoretic Schubert classes in $K(X)$, in the sense that
\[
\groth_u \cdot \groth_v = \sum_{w \in S_n} C_{u,v}^w  \groth_w,
\]
with the same structure coefficients as before (cf.\ \cite{Lascoux.Schutzenberger, Fulton.Lascoux}). Indeed more general \emph{$\beta$-Grothendieck polynomials} (introduced by S.~Fomin--A.~Kirillov \cite{Fomin.Kirillov}) play the analogous role with respect to the richer \emph{connective $K$-theory} of $X$ \cite{Hudson}, which, as shown by P.~Bressler--S.~Evens \cite{Bressler.Evens}, is the most general complex-oriented generalized cohomology theory in which the standard method of constructing Schubert classes is well-defined. 

In this paper, we use the philosophy of \cite{Assaf.Searles} to introduce the \emph{glide polynomials}, which refine the $\beta$-Grothendieck polynomials and form a new basis of polynomials. We provide a positive combinatorial formula for the expansion of $\beta$-Grothendieck polynomials in the glide basis, as well as positive combinatorial Littlewood-Richardson rules for the glide expansions of products of glide or $\beta$-Grothendieck polynomials. 

This paper is organized as follows. In Section~\ref{sec:grothendieck_and_glide_polynomials}, we first recall the Grothendieck and $\beta$-Grothendieck polynomials. We then introduce the basis of glide polynomials and give a positive combinatorial rule for expressing $\beta$-Grothendieck polynomials in this basis. Finally, we show that a specialization of the glide polynomials yields precisely the \emph{fundamental slide polynomials} of S.~Assaf--D.~Searles \cite{Assaf.Searles}, which play an analogous role in decomposing Schubert polynomials. 
In Section~\ref{sec:quasi}, we show the stable limits of glide polynomials (specialized to $\beta=1$) are the \emph{multi-fundamental quasisymmetric functions} of T.~Lam--P.~Pylyavskyy \cite{Lam.Pylyavskyy}, a basis of the ring of quasisymmetric functions. Moreover, the glide polynomials refining symmetric $\beta$-Grothendieck polynomials (i.e., those representing classes in Grassmannians) are a new basis of quasisymmetric polynomials and can be seen as (connective) $K$-theoretic analogues of I.~Gessel's \emph{fundamental quasisymmetric polynomials} \cite{Gessel}. We give a positive combinatorial formula for expressing symmetric $\beta$-Grothendieck polynomials in this basis, compacting the set-valued tableau formula of A.~Buch \cite{Buch}. In Section~\ref{sec:multiplication}, we extend a $K$-theoretic analogue of the shuffle product due to T.~Lam--P.~Pylyavskyy \cite{Lam.Pylyavskyy} and use it to present our Littlewood-Richardson rules.

\section{Grothendieck and glide polynomials}\label{sec:grothendieck_and_glide_polynomials}
Here, we recall the Grothendieck polynomials of A.~Lascoux--M.-P.~Sch\"utzenberger \cite{Lascoux.Schutzenberger} and the more general $\beta$-Grothendieck polynomials of S.~Fomin--A.~Kirillov \cite{Fomin.Kirillov}. We then introduce the glide polynomials as certain refinements.

\subsection{Grothendieck polynomials}
While the original definition of Grothendieck polynomials was in terms of divided difference operators, we will follow a more concretely combinatorial description based on work of various authors \cite{Billey.Jockusch.Stanley, Bergeron.Billey, Fomin.Kirillov, Knutson.Miller}. Indeed, we will describe first the more general $\beta$-Grothendieck polynomials introduced by S.~Fomin--A.~Kirillov \cite{Fomin.Kirillov}. 

The $\beta$-Grothendieck polynomials naturally represent Schubert classes in the \emph{connective $K$-theory} of $X$ \cite{Hudson} and specialize to the ordinary Grothendieck polynomials at $\beta = -1$. They moreover specialize at $\beta = 0$ to the \emph{Schubert polynomials}, representing the Schubert classes in the Chow ring of $X$.
While our interest is primarily in these two specializations, we will write most of our theorems for general $\beta$ as a convenient way to describe both theories simultaneously. We find that using general $\beta$ requires little extra complication beyond considering the $\beta = -1$ case. 

We now turn to defining $\beta$-Grothendieck polynomials.
A {\bf pipe dream} $P$ is a tiling of the fourth quadrant of the plane by {\bf crossing pipes} $\cross$ and {\bf turning pipes} $\turn$ that uses finitely-many crossing pipes. The lines of $P$, traveling from the $y$-axis to the $x$-axis, are called {\bf pipes}. We number the pipes by the absolute value of the $y$-coordinate of their left endpoint. In the case that no two pipes of $P$ cross each other more than once, we say $P$ is {\bf reduced}. For any pipe dream $P$, its {\bf reduction} $\reduct(P)$ is the reduced pipe dream obtained by replacing all but the southwestmost $\cross$ between each pair of pipes with $\turn$. Note that if $P$ is reduced, then $\reduct(P) = P$. 

The {\bf permutation} of a reduced pipe dream $P$ is the permutation given by the $x$-coordinates of the right endpoints of the pipes, while the permutation of a nonreduced pipe dream is the permutation of its reduction. The {\bf excess} $\ex(P)$ of a pipe dream $P$ is the number of $\cross$'s in $P$ minus the number of $\cross$'s in $\reduct(P)$. Let $\PD(w)$ denote the set of all pipe dreams for the permutation $w$, and let $\PD_e(w)$ denote the subset of pipe dreams with excess $e$, so that $\PD_0(w)$ denotes the subset of reduced pipe dreams. The {\bf weight} $\wt(P)$ of a pipe dream $P$ is the {\bf weak composition} (i.e., finite sequence of nonnegative integers) $(a_1, a_2, \dots)$, where $a_i$ records the number of $\cross$'s in the $i$th row of $P$ (from the top).

\begin{example}\label{ex:reduced_PDs}
The pipe dream
\[
P = \pipes{%
      & 1 & 2 & 3 & 4\\
      1 & \turn & \cross & \cross & \tail \\
      2 & \turn & \cross & \tail \\
      3 & \cross & \tail \\
      4 &  \tail }
\]
is not reduced since pipes $3$ and $4$ cross twice. Its reduction is the reduced pipe dream 
\[
\reduct(P) = \pipes{ & 1 & 2 & 3 & 4 \\ 
1 & \turn & \cross & \cross & \tail \\
2 & \turn &  \turn & \tail \\
3 &\cross & \tail \\
4 & \tail }
\]
 obtained by removing the second crossing between those pipes. Since $\reduct(P) \in \PD_0(1432)$, we have $P \in \PD_1(1432) \subset \PD(1432)$. The weight of $P$ is the weak composition $(2,1,1)$, while the weight of $\reduct(P)$ is $(2,0,1)$.
\end{example}

For $w \in S_n$, the {\bf $\beta$-Grothendieck polynomial} $\cgroth_w^{(\beta)}$ is the following generating function for pipe dreams of $w$:
\[
\cgroth_w^{(\beta)} := \sum_{P \in \PD(w)} \beta^{\ex(P)} x^{\wt(P)},
\]
where $x^a := x_1^{a_1}x_2^{a_2} \dots$. Here we treat $\beta$ as a formal parameter. 
Two specializations of $\cgroth_w^{(\beta)}$ are particularly significant: For $\beta = -1$, the {\bf Grothendieck polynomials} \[\groth_w := \cgroth_w^{(-1)}\] represent the Schubert classes in $K(X)$, while for $\beta = 0$, the {\bf Schubert polynomials} 
\[
\schub_w := \cgroth_w^{(0)} = \sum_{P \in \PD_0(w)} x^{\wt(P)}
\]
represent the Schubert classes in the Chow ring of $X$; this reflects the fact that Chow rings are isomorphic to the associated graded algebras of $K$-theory rings (at least after tensoring with $\mathbb{Q}$). Henceforth, we optionally drop $\beta$ from the notation, unless it is specialized to a particular value.

\subsection{Glide polynomials}

Given a weak composition $a$, the {\bf flattening} of $a$ is the (strong) composition $\flatten(a)$ obtained by deleting all zero terms from $a$. For example, $\flatten(0102) = 12$. 

Given weak compositions $a$ and $b$ of length $n$, say that $b$ {\bf dominates} $a$, denoted by $b \geq a$, if
\begin{equation*}
  b_1 + \cdots + b_i \ge a_1 + \cdots + a_i
\end{equation*}
for all $i=1,\ldots,n$. For example, $0120 \ge 0111$. Note that this partial ordering on weak compositions extends the usual dominance order on partitions.

Define a {\bf weak komposition} to be a weak composition where the positive integers may be colored arbitrarily black or red.
The {\bf excess} $\ex(a)$ of a weak komposition $a$ is the number of red entries in $a$. 



\begin{definition}
Let $a$ be a weak composition with nonzero entries in positions $n_j$.
The weak komposition $b$ is a {\bf glide} of $a$ if there exist nonnegative integers $i_1 < \dots < i_\ell$ such that we have
\begin{itemize}
\item $b_{i_j + 1} + \dots + b_{i_{j+1}} = \flatten(a)_{j+1} + \ex(b_{i_j+1}, \dots, b_{i_{j+1}})$,
\item $i_{j+1} \leq n_{j+1}$, and
\item $b_{i_j + 1}$ is black.
\end{itemize} 
\end{definition}

\begin{example}
Let $a=(0,1,0,0,0,3)$. The weak kompositions $(1,1,{\color{red}1},0,{\color{red}1},{\color{red}3})$ and $(1,{\color{red}1},0,2,0,{\color{red}2})$ are glides of $a$.
\end{example}

\begin{definition}
For a weak composition $a$ of length $n$, the {\bf glide polynomial} $\glide^{(\beta)}_a = \glide^{(\beta)}_a(x_1, \dots, x_n)$ is 
\[ \glide^{(\beta)}_a = \sum_b \beta^{\ex(b)} x_1^{b_1} \cdots x_n^{b_n}, \] where the sum is over all weak kompositions $b$ that are glides of $a$. As for $\cgroth^{(\beta)}_w$, we may drop $\beta$ from the notation, unless it is specialized to a particular value.
\end{definition}

\begin{example}\label{ex:glidepoly} We have
  \[\glide_{0102} = \mathbf{x}^{0102} + \mathbf{x}^{1002} + \mathbf{x}^{0120} + \mathbf{x}^{1020} + \mathbf{x}^{1200} + \mathbf{x}^{0111} + \mathbf{x}^{1011} + \mathbf{x}^{1101} + \mathbf{x}^{1110}+\] \[\beta\mathbf{x}^{0112} + \beta\mathbf{x}^{1012} + 2\beta\mathbf{x}^{1102} +  2\beta\mathbf{x}^{1120} +  \beta\mathbf{x}^{1021} +  \beta\mathbf{x}^{0121} + 3\beta\mathbf{x}^{1111} + \beta\mathbf{x}^{1210} + \beta\mathbf{x}^{1201} + \]
  \[ 2\beta^2\mathbf{x}^{1112} + 2\beta^2\mathbf{x}^{1121} + \beta^2\mathbf{x}^{1211},\]
  where $\mathbf{x}^b = x_1^{b_1}\ldots x_n^{b_n}$.
\end{example}

Let ${\rm Poly}_n := \mathbb{Z}[x_1, x_2, \dots, x_n]$ denote the ring of polynomials in $n$ variables.

\begin{theorem}\label{thm:basis}
The set \[ \{\beta^k \glide_a : k \in \mathbb{Z}_{\geq 0} \text{ and $a$ is a weak composition of length } n \}\] is an additive basis of the free $\mathbb{Z}$-module ${\rm Poly}_n[\beta]$. 

Hence $\{ \glide^{(-1)}_a : \text{$a$ is a weak composition of length } n \}$ is a basis of ${\rm Poly}_n$.
\end{theorem}
\begin{proof}
A monomial $m$ in ${\rm Poly}_n[\beta]$ is determined by a pair $(k, a)$, where $k \in \mathbb{Z}_{\geq 0}$ records the degree of $\beta$ in $m$ and $a$ is the weak composition of length $n$ that records the degrees of $x_1, \dots, x_n$ in $m$. Let $\mathcal{M}$ denote the set of such pairs $(k,a)$.

Define a total order on $\mathcal{M}$ by $(k, a) \succ (\ell, b)$ if
\begin{itemize}
\item $a$ contains strictly more $0$'s than $b$, 
\item $a$ and $b$ contain equal numbers of $0$'s and $b$ precedes $a$ in reverse lexicographic order, or
\item $a = b$ and $k > \ell$.
\end{itemize}
Now the $\prec$-leading term of $\glide_a$ is $\beta^0\mathbf{x}^a$. Hence if the $\prec$-leading term of $p \in {\rm Poly}_n[\beta]$ is $c_a  \beta^k \mathbf{x}^a$, then the $\prec$-leading term of \[p_2 := p - c_a \beta^k \glide_a\] is $c_b \beta^\ell \mathbf{x}^b$ for some $(\ell, b) \prec (k,a)$. Then \[p_3 := p_2 - c_b \beta^\ell \glide_b\] has $\prec$-leading term $c_d \beta^m \mathbf{x}^d$ for some $(m, d) \prec (\ell, b)$, etc. Since there are no infinite strictly $\prec$-decreasing sequences in $\mathcal{M}$, this process terminates with an expansion of $p$ as a finite sum of glide polynomials times powers of $\beta$, proving the first sentence of the theorem.

The second sentence of the theorem is immediate from the first.
\end{proof}

\subsection{Expanding $\beta$-Grothendieck polynomials in the glide basis}
By Theorem~\ref{thm:basis}, $\cgroth_w$ may be uniquely written as a sum of glide polynomials in the form
\[
\cgroth_w = \sum_{(k,a)} c_w^{(k,a)} \beta^k \glide_a,
\] where $c_w^{(k,a)} \in \mathbb{Z}$. This subsection is devoted to showing that these coefficients $c_w^{(k,a)}$ are in fact \emph{nonnegative} integers; we show this by giving an explicit positive combinatorial formula for $c_w^{(k,a)}$.

The following two notions extend definitions of S.~Assaf--D.~Searles \cite{Assaf.Searles} to non-reduced pipe dreams.
\begin{definition}
  For $P\in\PD(w)$, the {\bf destandardization of $P$}, denoted by $\destand(P)$, is the pipe dream constructed from $P$ as follows. For each row, say $i-1$, with no $\cross$ in the first column, if every $\cross$ in row $i-1$ lies strictly east of every $\cross$ in row $i$, then shift every $\cross$ in row $i-1$ southwest one position (if the westmost $\cross$ of row $i-1$ is immediately northeast of a $\cross$, then these two crosses merge during the shift). Repeat until no such row exists.
  \label{def:w-destand}  
\end{definition}

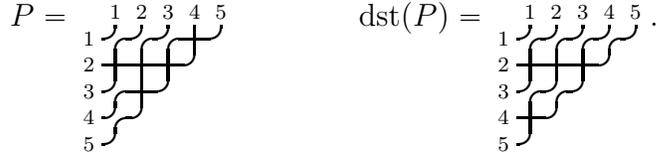
\begin{example}
\[
P= \pipes{ & 1 & 2 & 3 & 4 & 5 \\
1 & \turn & \turn & \turn & \cross & \tail \\
2 & \cross & \cross & \cross & \tail \\
3 & \turn & \cross & \tail \\
4 & \turn & \tail \\
5 & \tail } 
\qquad \qquad \destand(P) = \pipes{ & 1 & 2 & 3 & 4 & 5 \\
1 & \turn & \turn & \turn & \turn & \tail \\
2 & \cross & \cross & \cross & \tail \\
3 & \turn & \turn & \tail \\
4 & \cross & \tail \\
5 & \tail }.
\]
\end{example}

\begin{definition}
 A pipe dream is {\bf quasi-Yamanouchi} if the following is true for the westmost $\cross$ in every row: Either 
 \begin{itemize}
 \item[(1)]  it is in the westmost column, or
 \item[(2)] it is weakly west of some $\cross$ in the row below it.
\end{itemize}  Let $\QPD(w)$ denote the set of quasi-Yamanouchi pipe dreams for the permutation $w$ and let $\QPD_e(w)$ be the subset of those with excess $e$.
\end{definition}

\begin{example}\label{ex:QY_PDs}
The pipe dream $\reduct(P)$ of Example~\ref{ex:reduced_PDs} is not quasi-Yamanouchi, since the westmost $\cross$ in the top row is not in the first column and there is no $\cross$ in the row below. In the pipe dream $P$ of Example~\ref{ex:reduced_PDs} the westmost $\cross$ in the top row is weakly west of a $\cross$ in the second row. However the $\cross$ in the second row is neither in the first column nor weakly west of a $\cross$ in the third row. Hence $P$ is not quasi-Yamanouchi either. 

A quasi-Yamanouchi pipe dream for $1432$ is
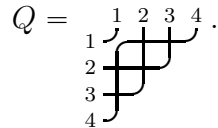
\[
Q = \pipes{ & 1 & 2 & 3 & 4 \\
1 & \turn & \cross & \cross & \tail \\
2 & \cross & \cross & \tail \\
3 & \cross & \tail \\
4 & \tail }.
\] (The reduction of $Q$ is formed by removing the $\cross$'s in the top row.)
\end{example}

The {\bf Lehmer code} $L(w)$ of a permutation $w$
is the weak composition whose $i$th term is the number of indices $j$ for which $i < j$ and $w(i) > w(j)$. For example, $L(146235) = (0, 2, 3, 0, 0, 0)$.

\begin{lemma}
  The destandardization map is well-defined and satisfies the following:
  \begin{enumerate}
  \item for $P \in \PD(w)$, $\destand(P) \in \QPD(w)$;
  \item for $P \in \PD(w)$, $\destand(P)=P$ if and only if $P \in \QPD(w)$;
  \item $\destand:\PD(w) \rightarrow \QPD(w)$ is surjective;
  \item $\destand:\PD(w) \rightarrow \QPD(w)$ is injective if and only if $w_i<w_{i+1}$ for all $i\ge w^{-1}(1)$.
  \end{enumerate}
  \label{lem:w-destand}
\end{lemma}

\begin{proof}  
Observe that if $P \in \PD(w)$, applying destandardization at row $i$ gives another pipe dream for $w$. The destandardization procedure terminates only when the quasi-Yamanouchi condition is satisfied, proving (1) and (2). Property (3) is immediate from (2). 

 For property (4), note that for any $w$, there is a reduced pipe dream $P_{L(w)}$ given by placing $L(w)_i$ $\cross$'s in row $i$, column $1$. Suppose $w$ has no descent after the $m$th position, where $m := w^{-1}(1)$. Then $P_{L(w)}$ has $\cross$'s in row $i$, column $1$ for all $i<m$, and no $\cross$'s in row $i$ for $i\ge m$. It is then immediate from the local moves connecting elements of $\PD_0(w)$ (\cite{Bergeron.Billey}) that every reduced pipe dream for $w$ has $\cross$'s in row $i$, column $1$ for all $i<m$, and no $\cross$'s in row $i$ for $i\ge m$. Thus, the same is true for all $P \in \PD(w)$ and hence $\destand(P) = P$ for all $P\in\PD(w)$.   
Conversely, if $w$ has a descent after the $m$th position, then by \cite[Lemma~3.12(4)]{Assaf.Searles}, the map $\destand : \PD_0(w) \to \QPD_0(w)$ is not injective, so certainly the extension $\destand : \PD(w) \to \QPD(w)$ is not injective.
\end{proof}

\begin{theorem}\label{thm:glide_expansion}
For $w$ any permutation, we have
\begin{equation*}
\cgroth_w =  \sum_{Q \in \QPD(w)} \beta^{\ex(Q)} \glide_{\wt(Q)}.
\end{equation*}
\end{theorem}
\begin{proof}
By Lemma~\ref{lem:w-destand}, it suffices to show that, for $Q \in \QPD(w)$, we have 
\begin{equation*}
\glide_{\wt(Q)} = \sum_{P \in \destand^{-1}(Q)} \beta^{\ex(P) - \ex(Q)}  x^{\wt(P)}.
\end{equation*}
By definition, 
\[ \glide_{\wt(Q)} = \sum_{b \text{ is a glide of } \wt(Q)} \beta^{\ex(b)} x_1^{b_1} \cdots x_n^{b_n}.\] 

For a pipe dream $P$, the {\bf colored weight} of $P$ is the weak komposition $\kwt(P)$ obtained by coloring the $i$th entry of $\wt(P)$ red if a $\cross$ can merge into the rightmost $\cross$ of the $i$th row of $P$ during application of $\destand$.
It is not hard to see that if $\destand(P)=Q$, then $\kwt(P)$ is a glide of $\wt(Q)$.

Conversely, we claim that given $Q \in \QPD(w)$, for every weak komposition $b$ that is a glide of $\wt(Q)$, there is a unique $P \in \PD(w)$ with $\kwt(P) = b$ such that $\destand(P) = Q$.  To construct this $P$ from $b$ and $Q$, for $j = 1,\ldots,n$, if $\wt(Q)_{j} = b_{i_{j-1} + 1} + \cdots + b_{i_{j}} - \ex(b_{i_{j-1} + 1}, \dots, b_{i_{j}})$, then, from east to west, shift the first $b_{i_{j-1} + 1} + \cdots + b_{i_j - 1}$ $\cross$'s northeast from row $j$ to row $j-1$ while leaving a copy of the leftmost of these moved $\cross$'s in place if $b_{i_j}$ is red, the first $b_{i_{j-1} + 1} + \cdots + b_{i_j - 2}$ $\cross$'s northeast from row $j-1$ to row $j - 2$ while leaving a copy of the leftmost of these moved $\cross$'s in place if $b_{i_j - 1}$ is red, and so on. This proves existence, and uniqueness follows from the lack of choice at each step.
\end{proof}

\subsection{Fundamental slide polynomials}

The fundamental slide basis of ${\rm Poly}_n$ was introduced by S.~Assaf-D.~Searles \cite{Assaf.Searles}, who applied it to the study of Schubert polynomials.
We say that a composition $b$ {\bf refines} a composition $a$ if $a$ can be obtained by summing consecutive entries of $b$, e.g.,~$(1,1,2,1)$ refines $(2,3)$ but $(1,2,1,1)$ does not.
 For a weak composition $a$ of length $n$, define the {\bf fundamental slide polynomial} $\Fund_{a} = \Fund_{a}(x_1,\ldots,x_n)$ by
  \begin{equation*}
    \Fund_{a} = \sum_{\substack{b \geq a \\ \flatten(b) \ \mathrm{refines} \ \flatten(a)}}  x_1^{b_1}\ldots x_n^{b_n}.
  \end{equation*}
  
\begin{example}\label{ex:slidepoly}
\[
  \Fund_{0102} = \mathbf{x}^{0102} + \mathbf{x}^{1002} + \mathbf{x}^{0120} + \mathbf{x}^{1020} + \mathbf{x}^{1200} + \mathbf{x}^{0111} + \mathbf{x}^{1011} + \mathbf{x}^{1101} + \mathbf{x}^{1110}.
\]  Notice that $\Fund_{0102} = \glide^{(0)}_{0102}$ (see Example~\ref{ex:glidepoly}).
\end{example}

\begin{theorem}
The fundamental slide polynomials are a specialization of the glide polynomials. More precisely, \[\Fund_a = \glide_a^{(0)}.\]
\end{theorem}
\begin{proof}
If $b$ is a glide of $a$ with excess $0$, then all entries of $b$ are black, so $b$ is a weak composition obtained by shifting or splitting the entries of $a$ to the left while preserving their relative order. Conversely, every such weak composition may be so obtained.
\end{proof}

\begin{remark}
Setting $\beta = 0$ in Theorem~\ref{thm:glide_expansion} recovers \cite[Theorem~3.13]{Assaf.Searles} for the fundamental slide expansion of Schubert polynomials.
\end{remark}

%

\section{Symmetric Grothendieck polynomials and quasisymmetric glide polynomials}\label{sec:quasi}
\subsection{Glide expansions of symmetric $\beta$-Grothendieck polynomials}
When $w$ is a Grassmannian permutation, i.e., $w$ has at most one descent, $\cgroth^{(\beta)}_w$ is a symmetric polynomial (with coefficients in $\mathbb{Z}[\beta]$). Let $n$ be the index of the rightmost nonzero entry of $L(w)$, or equivalently the position of the unique descent of $w$. Then the symmetric $\beta$-Grothendieck polynomial $\cgroth_w$ may be written as $\cgroth_\lambda(x_1,\ldots x_n)$ where $\lambda$ is the partition given by reading the nonzero entries of $L(w)$ in reverse. We identify the partition $\lambda$ with its corresponding Young diagram (in English orientation).

A {\bf set-valued tableau} of shape $\lambda$ is obtained by filling each box of the Young diagram $\lambda$ with a nonempty set of positive integers, subject to the condition that if a box is filled with a set $A$, then the smallest number in the box immediately to the right (respectively, immediately below) is at least as large as (respectively, strictly larger than) $\max (A)$. The {\bf weight} $\wt(T)$ of a set-valued tableau $T$ is the weak composition whose $i$th entry is the number of occurrences of $i$ in $T$. Let $|T|$ denote the sum of the entries of $\wt(T)$.

In \cite[Theorem 3.1]{Buch}, A. Buch expressed the monomial expansion of $\cgroth^{(-1)}_\lambda$ as a weighted sum of set-valued tableaux; this formula easily extends to the case of general $\beta$. Let $\SV_n(\lambda)$ denote the collection of all set-valued tableaux of shape $\lambda$ using labels from $\{1,\ldots , n\}$.

\begin{theorem}[{\cite{Buch}}]\label{thm:buch.symmetric}
\begin{equation*} 
\cgroth_\lambda(x_1,\ldots x_n) =  \sum_{T \in \SV_n(\lambda)} \beta^{|T|-|\lambda|}x^{\wt(T)},
\end{equation*}
where $|\lambda|$ denotes the number of boxes in $\lambda$.
\end{theorem}

In \cite[Definition~2.4]{Assaf.Searles}, S. Assaf--D. Searles gave a condition for a semistandard Young tableau to be \emph{quasi-Yamanouchi}, and used this to express the fundamental slide expansion of a Schur polynomial $s_\lambda(x_1,\ldots , x_n)$ in terms of quasi-Yamanouchi tableaux of shape $\lambda$. 
We extend this concept to set-valued tableaux in order to give a tableau formula for the glide expansion of a symmetric $\beta$-Grothendieck polynomial. 

\begin{definition}\label{def:QSV}
A set-valued tableau $T$ is {\bf quasi-Yamanouchi} if for all $i>1$, some instance of $i$ in $T$ is weakly left of some $i-1$ that is not in the same box.
\end{definition}

In the case there is only one entry per box, i.e., $T$ is a semistandard Young tableau, Definition~\ref{def:QSV} reduces to the definition of quasi-Yamanouchi tableaux from \cite[Definition~2.4]{Assaf.Searles}. 
For a weak composition $a$ of length $n$, let $\rev(a)$ be the weak composition of length $n$ obtained by reversing the entries of $a$.

\begin{theorem}\label{thm:symmetric}
For $\lambda$ any partition, we have
\begin{equation*} 
\cgroth_\lambda(x_1,\ldots x_n) =  \sum_{T \in \QSV_n(\lambda)} \beta^{|T|-|\lambda|}\glide_{\rev(\wt(T))}.
\end{equation*}
\end{theorem}
\begin{proof}
Fix $n$ and a partition $\lambda$, and let $w$ be the corresponding Grassmannian permutation.  
Define a map $\phi : \SV_n(\lambda) \rightarrow \PD(w)$ as follows. 
Given $T\in \SV_n(\lambda)$, flip $T$ upside-down, and place it in the fourth quadrant so that the boxes of $T$ are placed exactly over the crosses of the pipe dream $P_{L(w)}$ associated to the Lehmer code of $w$.
Then for each label $i$ of $T$, turn it into a cross and move it $i+r-n-1$ steps northeast, where $r$ is the index of the row in which the cross starts. This map $\phi$ is, up to convention, the bijection of \cite[Theorem~5.5]{Knutson.Miller.Yong}.


We now show that the restriction of $\phi$ to $\QSV_n(\lambda)$ is a bijection from $\QSV_n(\lambda)$ to $\QPD(w)$. Let $T\in \SV_n(\lambda)$. Notice that under $\phi$, labels $i$ in boxes of $T$ become crosses in row $n+1-i$ of $\phi(T)$. 

First suppose $T$ is quasi-Yamanouchi. Then for every $i$, some instance of $i$ is weakly left of some instance of $i-1$ in $T$ (and in a different box). By semistandardness the box containing this $i$ is strictly below the box containing this $i-1$. Therefore, the cross corresponding to this $i$ moves weakly fewer steps northeast than the cross corresponding to this $i-1$, so there is a cross in row $n+1 - i$ weakly west of a cross in row $n+2 - i$ in $\phi(T)$, satisfying the quasi-Yamanouchi condition on these two rows. Since $i$ was arbitrary, $\phi(T)$ is therefore quasi-Yamanouchi.

Now suppose $T$ is not quasi-Yamanouchi. Then for some $i>1$, all the $i$'s in $T$ are strictly right of all the $i-1$'s, except possibly for a unique box containing both an $i$ and an $i-1$. If a label $i$ in $T$ is to the right of another label $i$, then by semistandardness the first label is also weakly above the second; hence, the cross of $\phi(T)$ corresponding to this first $i$ is right of the cross of $\phi(T)$ corresponding to the second $i$. Since moreover there cannot be two instances of $i$ in the same column of $T$, it is therefore enough to check that the cross of $\phi(T)$ corresponding to the leftmost $i$ in $T$ is strictly east of the cross corresponding to the rightmost $i-1$ in $T$. If there is a box ${\sf b}$ of $T$ containing both $i$ and $i-1$, then ${\sf b}$ contains the leftmost $i$ and rightmost $i-1$. The cross $\cross_i$ of $\phi(T)$ corresponding to this $i$ sits immediately northeast (and thus strictly east) of the cross $\cross_{i-1}$ corresponding to this $i-1$. If there is no such box ${\sf b}$, then let ${\sf b}_i$ denote the box of the leftmost $i$ and let ${\sf b}_{i-1}$ denote the box of the rightmost $i-1$. By semistandardness, ${\sf b}_i$ is weakly above ${\sf b}_{i-1}$. So the cross $\cross_i$ corresponding to the $i \in {\sf b}_i$ moves strictly more steps northeast than the cross $\cross_{i-1}$ corresponding to the $i-1 \in {\sf b}_{i-1}$. Therefore $\cross_i$ is strictly right of $\cross_{i-1}$ in $\phi(T)$ and $\phi(T)$ is not quasi-Yamanouchi.

Since $\phi : \SV_n(\lambda) \rightarrow \PD(w)$ is a bijection and we have just shown $\phi^{-1}(\QPD(w)) = \QSV_n(\lambda)$, it follows that the restriction $\phi |_{\QSV_n(\lambda)} : \QSV_n(\lambda) \to \QPD(w)$ is well-defined and bijective.

Finally, if $T\in \QSV_n(\lambda)$, then it is clear that $\wt(\phi(T)) = \rev(\wt(T))$. The theorem now follows from Theorem~\ref{thm:glide_expansion}.
\end{proof}

\begin{example}
Let $w=13524$. Then $L(w) = (0,1,2,0,0)$, $n=3$ and the partition $\lambda$ corresponding to $w$ is $(2,1)$. We have
\begin{eqnarray}\nonumber
\cgroth_{13524} = \cgroth_{(2,1)}(x_1, x_2, x_3) = x_1^2x_2 + x_1^2x_3 + x_1x_2^2 + 2x_1x_2x_3 + x_1x_3^2 + x_2^2x_3 + x_2x_3^2 \\ \nonumber
+ \beta x_1^2x_2^2 + 3\beta x_1^2x_2x_3 + \beta x_1^2x_3^2 +3\beta x_1x_2^2x_3 +3\beta x_1x_2x_3^2 +\beta x_2^2x_3^2 \\ \nonumber
 +2 \beta^2 x_1^2x_2^2x_3 +2\beta^2 x_1^2x_2x_3^2 +2\beta^2 x_1x_2^2x_3^2 +\beta^3 x_1^2x_2^2x_3^2.
\end{eqnarray}

The elements of $\QSV_3(2,1)$ are
  \[
  \ytableausetup{boxsize=2.4em}
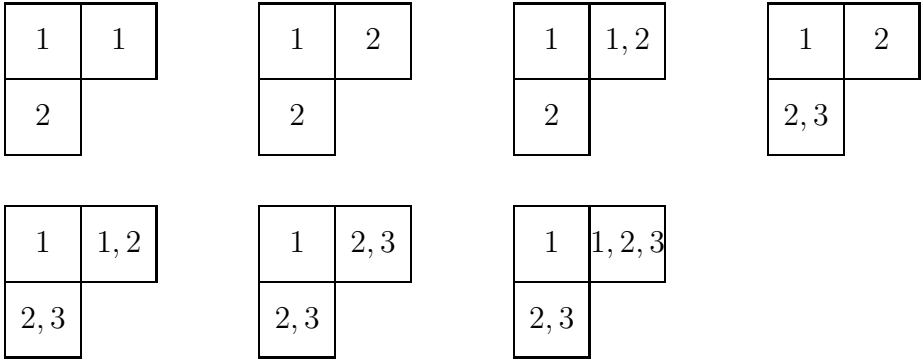
  \begin{array}{cccc}
  \ytableaushort{11,2} & \hspace{10mm} \ytableaushort{12,2} & \hspace{10mm} \ytableaushort{1{1,2},2} & \hspace{10mm} \ytableaushort{12,{2,3}} \\
  \medskip \\
\ytableaushort{1{1,2},{2,3}}  & \hspace{10mm} \ytableaushort{1 {2,3}, {2,3}} & \hspace{10mm} \ytableaushort{1 {1,2,3},{2,3}} &
  \end{array}
  \]
  
Rather than summing over the 27 elements of $\SV_3(2,1)$ to obtain $\cgroth_{13524}$, we may use Theorem~\ref{thm:symmetric} to sum over the $7$ elements of $\QSV_3(2,1)$, obtaining:
\begin{equation*}
\pushQED{\qed}
\cgroth_{13524} = \cgroth_{(2,1)}(x_1, x_2, x_3) = \glide_{012} + \glide_{021} +\beta \glide_{022} +\beta \glide_{121} +\beta^2 \glide_{122} +\beta^2 \glide_{221} +\beta^3 \glide_{222}. \qedhere \popQED
\end{equation*} \let\qed\relax
\end{example}

\subsection{Quasisymmetric polynomials and stable limits of glide polynomials}
A polynomial $f \in {\rm Poly}_n$  is {\bf quasisymmetric} if the coefficient of $x_{i_1}^{a_1} \dots x_{i_k}^{a_k}$ is equal to the coefficient of $x_{j_1}^{a_1} \dots x_{j_k}^{a_k}$ for any two strictly increasing sequences $i_1 < \cdots < i_k$ and $j_1 < \cdots j_k$. These polynomials were introduced by I.~Gessel in \cite{Gessel}, who used them in the study of $P$-partitions. We write ${\rm QSym}_n$ for the subspace of quasisymmetic polynomials in ${\rm Poly}_n$.  I.~Gessel also defined the {\bf fundamental basis} $\{ F_a \}$ of ${\rm QSym}_n$, indexed by compositions:
\[ 
F_a(x_1, \dots, x_n) = \!\!\! \sum_{\substack{b \text{ is a weak composition} \\ \flatten(b) \text{ refines } a}} \!\!\!  x^b.
\]



In \cite{Lam.Pylyavskyy}, T. Lam--P. Pylyavskyy introduced the \emph{multi-fundamental quasisymmetric functions} (defined below), which form a basis of the ring of quasisymmetric functions (in countably-many variables). The multi-fundamental quasisymmetric functions are a $K$-theoretic analogue of I. Gessel's  \cite{Gessel} basis of fundamental quasisymmetric functions, and have been further studied in \cite{Patrias}.

Let $S_1$ and $S_2$ be nonempty subsets of $\mathbb{Z}_{>0}$. Say that $S_1<S_2$ if $\max(S_1)<\min(S_2)$, and $S_1\le S_2$ if $\max(S_1)\le \min(S_2)$. For a strong composition $a$, let $\tilde{A}_a$ be the collection of all chains $\sigma=(S_1,\ldots S_{|a|})$ of nonempty subsets of positive integers such $S_i<S_{i+1}$ if there is some $k$ such that $a_1+\ldots + a_k = i$, and $S_i\le S_{i+1}$ otherwise. 

The {\bf multi-fundamental quasisymmetric function} $\tilde{L}_a(\mathbf{x}) = \tilde{L}_a(x_1,x_2,\ldots )$ is defined by
\[\tilde{L}_a(\mathbf{x}) = \sum_{\sigma \in \tilde{A}_a} x^{\wt(\sigma)},\]
where the $i$th entry of $\wt(\sigma)$ is the number of occurrences of $i$ in $\sigma$.

We now show that the multi-fundamental quasisymmetric functions are the stable limits of the glide polynomials (specialized to $\beta=1$). Let $0^m  a$ denote the weak composition obtained by prepending $m$ zeros to $a$. 

\begin{theorem}\label{thm:stable_limit}
For any weak composition $a$,
\[\lim_{m\to \infty} \glide^{(1)}_{0^m a} = \tilde{L}_{\flatten(a)}(\mathbf{x}).\]
\end{theorem}
\begin{proof}
We give a bijection between the glides indexing monomials in $\glide^{(1)}_{0^m a}(x_1,\ldots x_m)$ and the chains $\sigma \in \tilde{A}_a$ indexing monomials in the truncation $\tilde{L}_{\flatten(a)}(x_1, \ldots , x_m)$.

Let $\sigma \in \tilde{A}_a$, where $\sigma$ uses numbers in $\{1,\ldots , m\}$ only. Then the corresponding glide $b$ is simply the weight vector $\wt(\sigma)$, with entries colored as follows: If some $j$ appears in the same subset as some $i<j$, then $b_j$ is red. Otherwise it is black. For example, let $a=(0,0,0,0,3)$, so $\flatten(a)=(3)$. If $\sigma = (\{1,3\}, \{3,4\}, \{5\})$, then $b = (1,0,{\color{red}{2}},{\color{red}{1}}, 1)$.

For the reverse direction, let $b$ be a glide of $0^m a$ such that $b_i=0$ for $i>m$. Then $\sigma$ partitions the collection of $b_1$ ones, $b_2$ twos, etc., into a chain of nonempty subsets of $\mathbb{Z}_{>0}$. 
Suppose the first nonzero entry of $b$ is $b_j$. Then the first $b_j-1$ subsets in $\sigma$ are all singletons $\{j\}$, and the final $j$ is assigned to the $b_j$th subset. If the next nonzero entry, say $b_k$, of $b$ is black, then the $b_j$th subset is also the singleton $\{j\}$; now continue the process with $b_k$. If, on the other hand, $b_k$ is red, then assign a $k$ to the $b_j$th subset and continue in this manner.
For example, let $a=(0,0,0,0,3)$, so $\flatten(a)=(3)$. If $b=(1,0,2,{\color{red}{1}},{\color{red}{1}})$, then $\sigma = (\{1\}, \{3\}, \{3,4,5\})$, while if $b=(1,0,{\color{red}{2}},1,{\color{red}{1}})$, then $\sigma = \{1,3\}, \{3\}, \{4,5\}$.

These maps are clearly mutually inverse.
\end{proof}

We say a polynomial in ${\rm Poly}_n[\beta]$ is {\bf quasisymmetric} if it lies in ${\rm QSym}_n[\beta]$. Define a weak composition $a$ to be {\bf quasiflat} if the nonzero entries of $a$ occur in an interval. In \cite{Assaf.Searles}, it was shown that $\Fund_a$ is quasisymmetric in $x_1,\ldots , x_n$ if and only if $a$ is quasiflat with last nonzero term in postition $n$, and that moreover in this case $\Fund_a = F_{\flatten(a)}(x_1,\ldots , x_n)$. Since $\Fund_a = \glide^{(0)}_a$, this immediately implies that $\glide^{(\beta)}_a$ is not quasisymmetric if $a$ is not quasiflat.

Using the glide polynomials, we define a family of polynomials $G^{(\beta)}_a$ indexed by strong compositions. 

\begin{definition}
Given a strong composition $a$, let the {\bf quasisymmetric glide} be
\[G^{(\beta)}_a(x_1,\ldots , x_n) = \begin{cases}
  \glide^{(\beta)}_{0^{n-\ell(a)} a} & \mbox{if $\ell(a)\le n$} \\
  0 & \mbox{otherwise}.
  \end{cases}\]
\end{definition}


The fact that $G^{(\beta)}_a$ is quasisymmetric, and that indeed $G^{(1)}_a$ is a truncation of $\tilde{L}_a$, follows immediately from the bijection in the proof of Theorem~\ref{thm:stable_limit} and the fact that no nonzero entry precedes a zero entry in $0^{n-\ell(a)} a$. Nonetheless, we will use combinatorics of glides to give a direct proof of quasisymmetry, proving moreover that $G_a$ expands positively in the basis of fundamental quasisymmetric polynomials. Define $\flatten(b)$ for a weak komposition $b$ to be the strong composition given by deleting all $0$ entries of $b$ and forgetting the coloring.

\begin{definition}
Let $a$ be a weak composition.
A glide $b$ of $a$ is {\bf unsplit} if 
\begin{itemize}
\item $b$ has the same number of nonzero black entries as $a$ does and
\item no $0$ in $b$ is right of a nonzero entry.
\end{itemize}

\end{definition}

\begin{theorem}\label{thm:quasisymmetric}
For a strong composition $a$ with $ \ell(a) \leq n$,
\[G_a(x_1,\ldots , x_n) = \sum_b \beta^{\ex(b)} F_{\flatten(b)}(x_1,\ldots , x_n), \] 
where the sum is over the unsplit glides $b$ of 
$0^{n-\ell(a)}  a$.
\end{theorem}

\begin{proof}
Suppose that $c$ is a glide of $0^{n-\ell(a)} a$. Observe that the following local operations on the weak komposition $c$ produce another glide of $0^{n-\ell(a)} a$:
\begin{enumerate}
\item replacing the subword $0k$ by $k0$,
\item replacing the subword $0\color{red}{k}$ by ${\color{red}{k}}0$,
\item replacing the subword $0k$ by $ij$ with $i+j = k$,
\item replacing the subword $0\color{red}{k}$ with $i\color{red}{j}$ with $i+j = k$.
\end{enumerate}

Let $b$ be a unsplit glide of $0^{n-\ell(a)} a$. It is clear by repeated application of (1)--(4) that all monomials of $F_{\flatten(b)}(x_1,\ldots , x_n)$ appear from glides of $0^{n-\ell(a)} a$.

Now suppose $b$ and $b'$ are distinct unsplit glides of $0^{n-\ell(a)} a$. We need to ensure repeated application of (1)--(4) to $b$ and $b'$ yields disjoint sets of glides of $0^{n-\ell(a)} a$. Since (1)--(4) preserve the number of red entries, we may assume that $b$ and $b'$ both have $r$ red entries. For any weak komposition $c$, let $R_c$ denote the strong composition whose $i$th entry is the sum of the entries of $c$ that are strictly right of the $(i-1)$th red entry and weakly left of the $i$th red entry. Clearly, if $d$ is obtained from $c$ by any of (1)--(4), then $R_c = R_d$. It remains to note that $R_b \neq R_{b'}$.

Finally, suppose $c$ is a glide of $0^{n-\ell(a)} a$. We need to show $c$ can be obtained from some unsplit glide $b$ of $0^{n-\ell(a)} a$ by repeated application of (1)--(4). By definition of glides, there exists a unique  sequence of nonnegative integers $i_1 < \dots < i_\ell$ such that
\begin{itemize}
\item $c_{i_j + 1} + \dots + c_{i_{j+1}} = \flatten(a)_{j+1} + \ex(c_{i_j+1}, \dots, c_{i_{j+1}})$,
\item $i_{j+1} \leq n_{j+1}$, and 
\item $c_{i_j + 1}$ is black.
\end{itemize} 
In each block $(c_{i_j + 1}, \dots, c_{i_{j+1}})$, shift and combine entries to the right as much as possible by the inverses of (1)--(4). Concatenate the results in order into a new weak komposition $c'$.
Then push all entries of $c'$ as far right as possible by the inverses of (1) and (2). The result $b$ is a glide of $0^{n-\ell(a)} a$, since all entries of $0^{n-\ell(a)} a$ are themselves as far right as possible. Since $b$ has exactly one black entry for each block of $c$, $b$ is a unsplit glide of $0^{n-\ell(a)} a$, and since we obtained $b$ from $c$ by applying only the inverses of (1)--(4), $c$ is associated to the unsplit glide $b$.
\end{proof}

\begin{example}
\pushQED{\qed}
Let $a = (1,2)$ and $\mathbf{x} = (x_1, x_2, x_3, x_4)$. Then Theorem~\ref{thm:quasisymmetric} gives that 
\begin{align*}
G_{(1,2)}(\mathbf{x}) = &F_{(1,2)}(\mathbf{x}) + 2 \beta F_{(1,1,2)}(\mathbf{x}) + \beta F_{(1,2,1)}(\mathbf{x}) \\ 
&+ 3 \beta^2  F_{(1,1,1,2)}(\mathbf{x}) + 2 \beta^2  F_{(1,1,2,1)}(\mathbf{x})  + \beta^2  F_{(1,2,1,1)}(\mathbf{x}),
\end{align*} because the unsplit glides of $(0,0,1,2)$ are \[
\begin{array}{ccccc}
(0,0,1,2) & (0,1, {\color{red}1},2) & (0,1,1,{\color{red}2}) & (0,1,2,{\color{red}1}) & (1,{\color{red}1},  {\color{red}1},2) \\ (1,{\color{red}1},1,{\color{red} 2}) & (1,1,{\color{red}1},{\color{red}2}) & (1,{\color{red}1},2,{\color{red}1}) & (1,1,{\color{red}2},{\color{red}1}) & (1,2,{\color{red}1},{\color{red}1})
\end{array}. \qedhere \popQED \] \let\qed\relax
\end{example}

\begin{corollary}\label{cor:lowestisfundamental}
The fundamental quasisymmetric polynomials are a specialization of the quasisymmetric glide polynomials. More precisely, 
\[F_a(x_1,\ldots , x_n) = G^{(0)}_a(x_1,\ldots , x_n).\]
\end{corollary}

\begin{remark}
Our Theorem~\ref{thm:quasisymmetric} (at $\beta = 1$) is a finite-variable analogue of \cite[Theorem~5.12]{Lam.Pylyavskyy}, which is instead expressed in the language of injective order-preserving maps between chain posets. 
\end{remark}

\begin{corollary}\label{cor:Gexpansion}
\begin{equation}\nonumber 
\cgroth_\lambda(x_1,\ldots x_n) =  \sum_{T \in \QSV_n(\lambda)} \beta^{|T|-|\lambda|}G_{\rev(\wt(T))}(x_1,\ldots , x_n).
\end{equation}
\end{corollary}
\begin{proof}
By definition, if a quasi-Yamanouchi tableau $T$ uses the label $i>1$, it must also use the label $i-1$. Hence $\wt(T)$ is a strong composition (up to trailing $0$s). 
The corollary is then immediate from Theorem~\ref{thm:symmetric}.
\end{proof}

Specializing Corollary~\ref{cor:Gexpansion} to $\beta = 0$ recovers \cite[Theorem~2.7]{Assaf.Searles}, a rephrasing of I.~Gessel's celebrated expression \cite{Gessel}  for writing a {\bf Schur polynomial} $s_\lambda := \cgroth_\lambda^{(0)}$ as a sum of fundamental quasisymmetric polynomials. Specializing instead to $\beta = 1$ essentially gives an alternate formulation of (a special case of) \cite[Theorem~5.6]{Lam.Pylyavskyy} about expansions into multi-fundamental quasisymmetric functions.

\begin{theorem}\label{thm:quasibasis}
The set $\{\beta^k G_a : k \in \mathbb{Z}_{\geq 0} \text{ and } \ell(a)\le n \}$ is a basis of the ring of quasisymmetric polynomials ${\rm QSym}_n[\beta]$. Hence $\{G^{(-1)}_a : \ell(a)\le n \}$ is a basis of ${\rm QSym}_n$.
\end{theorem}
\begin{proof}
The map $a \mapsto 0^{n-\ell(a)} a$ is injective when $\ell(a)\le n$, and by Theorem~\ref{thm:basis} the polynomials $\{ \beta^k \glide_a \}$ are linearly independent. Hence $\{\beta ^k G_a : k \in \mathbb{Z}_{\geq 0} \text{ and } \ell(a)\le n \}$ is linearly independent. Since $G_a^{(0)} = F_a$ by Corollary~\ref{cor:lowestisfundamental} and $\{F_a : \ell(a)\le n \}$ is a basis of ${\rm QSym}$, the set $\{\beta ^k G_a : k \in \mathbb{Z}_{\geq 0} \text{ and } \ell(a)\le n \}$ spans ${\rm QSym}_n[\beta]$. Thus it is a basis of ${\rm QSym}_n[\beta]$. The second sentence of the theorem is then immediate.
\end{proof}

Putting together results of this section and the previous, we have the following relationships between bases of $\QSym_n$ and $\Poly_n$. Here upward arrows represent a lifting from quasisymmetric polynomials to polynomials, and rightward arrows represent a lift from ordinary cohomology to connective $K$-theory.

\[
\begin{tikzcd}
\{\Fund_b\} \subset \Poly_n \arrow[dashed]{r} \text{(\cite{Assaf.Searles})}  & \{\glide_b\} \subset \Poly_n \text{\phantom{(\cite{Assaf.Searles})}}  \\ 
\{F_a\} \subset \QSym_n \text{(\cite{Gessel})} \arrow[dashed]{r} \arrow[hookrightarrow]{u}  & \{G_a\} \subset \QSym_n \arrow[hookrightarrow]{u} \text{(cf.~\cite{Lam.Pylyavskyy})}  \\
\end{tikzcd}
\]

%


\section{Multiplication of glide polynomials}\label{sec:multiplication}
By Theorem~\ref{thm:basis}, the glide polynomials form a basis of ${\rm Poly}_n[\beta]$. Hence the product of two glide polynomials can be written uniquely as a sum of glide polynomials times powers of $\beta$. In this section, we show that this sum involves only \emph{positive} coefficients. We give an explicit positive combinatorial formula for these structure constants, extending the rule of S.~Assaf and D.~Searles~\cite[Theorem~5.11]{Assaf.Searles} for the multiplication of fundamental slide polynomials. Our rule also essentially restricts to \cite[Proposition~5.9]{Lam.Pylyavskyy} in the quasisymmetric (and $\beta = 1$) case, though we have some additional complexity related to having finitely-many variables. 

\subsection{The genomic shuffle product}
Here we give a reformulation of the \emph{multishuffle product} of \cite{Lam.Pylyavskyy}, a $K$-theoretic generalization of the \emph{shuffle product} of S.~Eilenberg--S.~Mac~Lane~\cite{Eilenberg.MacLane}. This reformulation is necessary for the statement of our Littlewood-Richardson rule. In this reformulation, we refer to the multishuffle product as the \emph{genomic shuffle product} because of resemblances to the \emph{genomic tableau} theory for (torus-equivariant) $K$-theoretic Schubert calculus introduced in \cite{Pechenik.Yong:KT} and further expounded in \cite{Pechenik.Yong:genomic}.

First we recall the classical shuffle product of S.~Eilenberg--S.~Mac~Lane. Let $A = A_1 A_2 \ldots A_n$ and $B = B_1 \ldots B_m$ be words on disjoint alphabets $\mathcal{A}$ and $\mathcal{B}$, respectively. The {\bf shuffle product} $A \shuffle B$ of $A$ and $B$ is the set of all permutations of the concatenation $AB$ such that the subword on the alphabet $\mathcal{A}$ is $A$ and the subword on the alphabet $\mathcal{B}$ is $B$.

\begin{example}\label{ex:shuffle_product}
The shuffle product of $331$ and $62$ is the set 
\[ \pushQED{\qed}
331 \shuffle 62 = \{ 62331, 63231, 63321, 63312, 36231, 36321, 36312, 33621, 33612, 33162 \}. \qedhere \popQED 
\] \let\qed\relax 
\end{example}

Add a superscript to each letter of $A$ so that, if $A_h$ is the $j$th instance of $i$ in $a$ (counting from left to right), it becomes $i^j$. Write $A^\gen$ for this superscripted version of $A$. Add superscripts to $B$ to obtain $B^\gen$ in the same way. For an alphabet $\mathcal{A}$, let $\mathcal{A}^\gen$ denote the set of symbols $i^j$, where $i \in \mathcal{A}$ and $j \in \mathbb{Z}_{>0}$. 
For $A$ a word in $\mathcal{A}^\gen$, a {\bf genotype}\footnote{See \cite{Pechenik.Yong:genomic} for motivation of this terminology.} is given by deleting all superscripts from any subword obtained by deleting all but one instance of each symbol $i^j$.
Let $A$ and $B$ be words in the alphabets $\mathcal{A}$ and $\mathcal{B}$, respectively. The {\bf genomic shuffle product} $A \shuffle_\gen B$ of $A$ and $B$ is the set of all words in the alphabet $(\mathcal{A} \sqcup \mathcal{B})^\gen$ such that
\begin{itemize}
\item if $i^j$ appears to the left of $i^k$, then $j \leq k$;
\item no two instances of $i^j$ are consecutive;
\item every genotype is an element of $A^\gen \shuffle B^\gen$.
\end{itemize}


\begin{remark}
The original definition of $\shuffle_\gen$ by T.~Lam--P.~Pylyavskyy \cite{Lam.Pylyavskyy} avoids reference to genotypes. We will need this language of genotypes to formulate some extra relations in dominance order that are necessary for describing the more-general structure constants of the glide basis. In the quasisymmetic case, if one works in countably-many variables, one may simplify our Littlewood-Richardson rule to coincide with their \cite[Proposition~5.9]{Lam.Pylyavskyy}.
\end{remark}

\begin{example}\label{ex:genomic_shuffle_product}
The genomic shuffle product $331 \shuffle_\gen 62$ is an infinite set of words, but contains finitely many words of any fixed length. It contains the $10$ words of length $5$ that are in $331 \shuffle 62$ (but with superscripted $1$'s on every letter, except the second $3$ which has a superscripted $2$), together with $35$ words of length $6$, $81$ words of length $7$, $154$ words of length $8$, and many longer words. The words of length $6$ in $331 \shuffle_\gen 62$ are
\[\begin{array}{ccccc}
6^1 3^1 2^1 3^1 3^2 1^1 & 3^1 6^1 2^1 3^1 3^2 1^1 & 3^1 6^1 3^1 2^1 3^2 1^1 & 3^1 6^1 3^1 3^2 2^1 1^1 & 3^1 6^1 3^1 3^2 1^1 2^1 \\ 
6^1 3^1 3^2 2^1 3^2 1^1 & 3^1 6^1 3^2 2^1 3^2 1^1 & 3^1 3^2 6^1 2^1 3^2 1^1 & 3^1 3^2 6^1 3^2 2^1 1^1 & 3^1 3^2 6^1 3^2 1^1 2^1 \\
6^1 3^1 3^2 1^1 2^1 1^1 & 3^1 6^1 3^2 1^1 2^1 1^1 & 3^1 3^2 6^1 1^1 2^1 1^1 & 3^1 3^2 1^1 6^1 2^1 1^1 & 3^1 3^2 1^1 6^1 1^1 2^1 \\
3^1 3^2 6^1 1^1 6^1 2^1 & 3^1 6^1 3^2 1^1 6^1 2^1 & 3^1 6^1 3^2 6^1 1^1 2^1 & 3^1 6^1 3^2 6^1 2^1 1^1 & 6^1 3^1 3^2 1^1 6^1 2^1 \\
6^1 3^1 3^2 6^1 1^1 2^1 & 6^1 3^1 3^2 6^1 2^1 1^1 & 6^1 3^1 6^1 3^2 1^1 2^1 & 6^1 3^1 6^1 3^2 2^1 1^1 & 6^1 3^1 6^1 2^1 3^2 1^1 \\
3^1 3^2 6^1 2^1 1^1 2^1 & 3^1 6^1 3^2 2^1 1^1 2^1 & 3^1 6^1 2^1 3^2 1^1 2^1 & 3^1 6^1 2^1 3^2 2^1 1^1 & 6^1 3^1 3^2 2^1 1^1 2^1 \\
6^1 3^1 2^1 3^2 1^1 2^1 & 6^1 3^1 2^1 3^2 2^1 1^1 & 6^1 2^1 3^1 3^2 1^1 2^1 & 6^1 2^1 3^1 3^2 2^1 1^1 & 6^1 2^1 3^1 2^1 3^2 1^1 
\end{array} \]
The two genotypes of $6^1 3^1 2^1 3^1 3^2 1^1$ are $63231$ and $62331$. 
\end{example}

\subsection{The glide product on weak compositions}

Let $S$ be a sequence of words in the alphabet $\mathcal{A}$ and let $\mathcal{B} \subseteq \mathcal{A}$ be a subalphabet. Then the {\bf $\mathcal{B}$-composition} ${\sf Comp}_{\mathcal{B}}(S)$ of $S$ is the weak composition whose $i$th coordinate is the number of letters of $\mathcal{B}$ in the $i$th word of $S$.
If $\mathcal{B} = \mathcal{A}$, we may drop $\mathcal{B}$ from the notation.

Order the alphabet $\mathbb{Z}_{>0}^\gen$ lexicographically; that is, $i^j < k^\ell$ if either $i < k$ or else $i=k$ and $j < \ell$. If $C$ is a word in $\mathbb{Z}_{>0}^\gen$, its {\bf run structure} ${\sf Runs}(C)$ is the sequence of successive maximally increasing runs of the symbols $i^j$ read from left to right. 
A {\bf genotype} of ${\sf Runs}(C)$  is given by deleting all superscripts from a sequence that comes from deleting all but one instance of each symbol $i^j$ in ${\sf Runs}(C)$. In particular, a genotype $G$ of ${\sf Runs}(C)$ is a sequence of (possibly empty) words in the alphabet $\mathbb{Z}_{>0}$.  

\begin{example}
Let $C = 6^1 3^1 6^1 3^2 1^1 2^1$. Then the run structure of $C$ is ${\sf Runs}(C) = (6^1, 3^1 6^1, 3^2, 1^1 2^1)$ and so ${\sf Comp}({\sf Runs}(C)) = (1,2,1,2)$. There are two genotypes of ${\sf Runs}(C)$, namely $G_1 = (6, 3, 3, 1 2)$ and $G_2 = (\epsilon , 3 6, 3, 1 2)$, where $\epsilon$ denotes the empty word. If $\mathcal{B}$ denotes the alphabet of even positive integers, then ${\sf Comp}_\mathcal{B}(G_1) = (1,0,0,1) $ and ${\sf Comp}_\mathcal{B}(G_2) = (0,1,0,1)$.
\end{example}

\begin{definition}
Let $a, b$ be weak compositions of length $n$. Let $A$ and $B$ be the words $A := (2n-1)^{a_1} \cdots (3)^{a_{n-1}} (1)^{a_n}$ and $B := (2n)^{b_1} \cdots (4)^{b_{n-1}} (2)^{b_n}$. Define the {\bf genomic shuffle set} ${\rm GSS}(a,b)$ of $a$ and $b$ by
\[
{\rm GSS}(a,b) := \{ C \in A \shuffle_\gen B : \text{for every genotype $G$ of ${\sf Runs}(C)$, ${\sf Comp}_\mathcal{A}(G) \geq a$ and ${\sf Comp}_\mathcal{B}(G) \geq b$}\},
\]
where $\mathcal{A}, \mathcal{B}$ respectively denote the alphabets of odd and even positive integers.
\end{definition}

\begin{example}\label{ex:genomic_shuffle_set}
Let $a = 021$ and $b = 101$. Then $A = 331$ and $B = 62$. The genomic shuffle product $331 \shuffle_\gen 62$ is (partially) described in Example~\ref{ex:genomic_shuffle_product}. We have ${\rm GSS}(021,101) = 
\left\{ \small{ \begin{array}{ccc}
6^1 2^1 3^1 3^2 1^1 & 6^1 2^1 3^1 3^2 1^1 2^1 & 3^1 6^1 2^1 3^1 3^2 1^1 2^1 \\
6^1 3^1 3^2 1^1 2^1 & 3^1 6^1 2^1 3^1 3^2 1^1  & 3^1 3^2 6^1 2^1 3^2 1^1 2^1 \\
3^1 6^1 2^1 3^2 1^1 & 3^1 6^1 3^1 3^2 1^1 2^1 & 3^1 3^2 6^1 1^1 2^1 1^1 2^1 \\
3^1 6^1 3^2 1^1 2^1 &3^1 3^2 6^1 2^1 3^2 1^1 & \\
3^1 3^2 6^1 2^1 1^1 & 3^1 3^2 6^1 3^2 1^1 2^1 & \\
3^1 3^2 6^1 1^1 2^1 & 3^1 6^1 2^1 3^2 1^1 2^1 & \\
& 3^1 3^2 6^1 2^1 1^1 2^1 & \\
& 3^1 3^2 6^1 1^1 2^1 1^1 &
\end{array} } \right\}.
$
\end{example}

Note that, while $A \shuffle_\gen B$ is usually an infinite set, ${\rm GSS}(a,b)$ is necessarily finite, since certainly no element of ${\rm GSS}(a,b)$ can have length more than $n \cdot (|a| + |b|)$. (We will significantly improve this upper bound in Corollary~\ref{cor:GSS_max_length}.)

\begin{remark}
Although it is convenient to define ${\rm GSS}(a,b)$ by a condition on all genotypes, in fact it is sufficient to verify this condition on a particular `worst' genotype. Specifically let $\hat{G}$ be the genotype of ${\sf Runs}(C)$ obtained by preserving the rightmost instance of each letter and deleting the others. Then $\hat{G}$ satisfies the desired dominance conditions if and only if every genotype of ${\sf Runs}(C)$ does.
\end{remark}

\begin{definition}
Let $a, b$ be weak compositions of length $n$. 
For $C \in {\rm GSS}(a,b)$, let ${\sf BumpRuns}(C)$ denote the unique dominance-minimal way to insert words of length $0$ into ${\sf Runs}(C)$ while preserving ${\sf Comp}_\mathcal{A}(G) \geq a$ and ${\sf Comp}_\mathcal{B}(G) \geq b$ for every genotype $G$ of ${\sf BumpRuns}(C)$. 
The {\bf glide product} $a \shuffle_\gen b$ of $a$ and $b$ is the \emph{multiset} of weak compositions 
\[
a \shuffle_\gen b := \{ {\sf Comp}({\sf BumpRuns}(C)) : C \in {\rm GSS}(a,b) \}.
\]
\end{definition}

\begin{theorem}\label{thm:structure_constants_for_glide polynomials}
For weak compositions $a$ and $b$ of length $n$, we have
\[
\glide_a \glide_b = \sum_c \beta^{|c|-|a|-|b|} g_{a,b}^c \glide_c,
\]
where $g_{a,b}^c$ denotes the multiplicity of $c$ in the glide product $a \shuffle_\gen b$.
\end{theorem}
\begin{proof}
For simplicity, we explicitly prove the theorem for the specialization $\beta=1$. It is clear that if the theorem is true for $\beta=1$, then it is true for general $\beta$.

Given a word $C\in {\rm GSS}(a,b)$, let $\overline{C}$ be the word in the alphabet $\mathbb{Z}_{>0}^\gen \cup \{ | \}$ obtained by inserting $|$'s into $C$ to separate the elements of ${\sf BumpRuns}(C)$.
For example, let $C = 3^1 3^2 6^1 1^1 2^1$ from ${\rm GSS}(021,101)$ in Example~\ref{ex:genomic_shuffle_set}. Then ${\sf Runs}(C) = (3^1 3^2 6^1, 1^1 2^1)$, and ${\sf BumpRuns}(C) = (3^1 3^2 6^1, \epsilon, 1^1 2^1)$. Hence $\overline{C} = 3^1 3^2 6^1 || 1^1 2^1$, where the $|$'s reflect the locations of the commas in ${\sf BumpRuns}(C)$. 

Let ${\sf shift}(\overline{C})$ denote the set of all words that can be formed from $\overline{C}$ by optionally replacing any letter from $\mathbb{Z}_{>0}^\gen$ with a nonempty string of copies of that letter, and by moving $|$'s to the right, such that $i^j < k^\ell$ whenever $i^j$ and $k^\ell$ are consecutive.

For example, if $\overline{C} = 3^1 3^2 6^1 || 1^1 2^1$,  then the elements of ${\sf shift}(\overline{C})$ are

$\begin{array}{ccccc}
3^1 3^2 6^1 || 1^1 2^1 & 3^1 3^2 6^1 | 1^1 | 2^1 & 3^1 3^2 6^1 | 1^1 2^1 | & 3^1 3^2 6^1 | 1^1 | 1^1 2^1 & 3^1 3^2 6^1 | 1^1 2^1 | 2^1. \\
\end{array}$

Define a set \[\overline{{\rm GSS}(a,b)} := \bigcup_{C \in {\rm GSS}(a,b)} {\sf shift}(\overline{C}) .\] 
Let $M(a,b)$ denote the set of ordered pairs $(a', b')$ of weak kompositions such that $a'$ is a glide of $a$ and $b'$ is a glide of $b$. Then the elements of $M(a,b)$ obviously correspond to the monomials in the product $\glide^{(1)}_a \glide^{(1)}_b$. We claim a bijection between $\overline{{\rm GSS}(a,b)}$ and $M(a,b)$; in particular, the elements of $\overline{{\rm GSS}(a,b)}$ represent the monomials appearing in $\glide^{(1)}_a \glide^{(1)}_b$.

Given an element $D\in \overline{{\rm GSS}(a,b)}$, let ${\sf Seq}(D)$ be the sequence of maximal consecutive subwords in the alphabet $\mathbb{Z}_{>0}^\gen$.
One then recovers an element $(a', b' ) \in M(a,b)$ by \[(a', b' ) = \Big( {\sf Comp}_{\mathcal{A}^\gen}\big({\sf Seq}(D)\big), {\sf Comp}_{\mathcal{B}^\gen}\big({\sf Seq}(D)\big) \Big),\] where we color $a_i'$ (respectively, $b_i'$) red if and only if the $i$th element of ${\sf Seq}(D)$ contains a letter $i^j \in \mathcal{A}^\gen$ (respectively, $i^j \in \mathcal{B}^\gen$) that also appears in a previous element of ${\sf Seq}(D)$. 

For example, $3^1 3^2 6^1 || 1^1 2^1$ maps to $((2,0,1),  (1,0,1))$, and $3^1 3^2 6^1 | 1^1 | 1^1 2^1$ maps to $((2,1,{\color{red}1}), (1,0,1))$. 


Given an element $(a', b') \in M(a,b)$, create $D\in \overline{{\rm GSS}(a,b)}$ as follows. The first run of $D$ is the first $a'_1$ letters of $A^\gen$ followed by the first $b'_1$ letters of $B^\gen$, sorted into increasing order, then the second run is the next $a'_2$ letters of $A^\gen$ followed by the next $b'_2$ letters of $B^\gen$, sorted into increasing order, etc, with the exception that whenever you see a red entry in $a'$ (respectively $b'$), the corresponding run of $D$ has a copy of the most-recently placed letter of $A^\gen$ (respectively $B^\gen$).  

For example, if $A^\gen = 3^1 3^2 1^1$, $B^\gen = 6^1 2^1$, then $(a',b')=((2,1,0),(1,1,{\color{red}1}))$ maps to $3^1 3^2 6^1 | 1^1 2^1 | 2^1$.

It is clear that these two maps are mutually inverse. Hence the elements of $\overline{{\rm GSS}(a,b)}$ represent the set of monomials in $\glide^{(1)}_a \glide^{(1)}_b$. By construction, for any $C\in {\rm GSS}(a,b)$, the monomials associated to the elements of ${\sf shift}(\overline{C})$ together comprise the glide polynomial $\glide^{(1)}_{{\sf Comp}({\sf BumpRuns}(C))}$. Continuing the running example of $C = 3^1 3^2 6^1 1^1 2^1$, the monomials corresponding to elements of ${\sf shift}(\overline{C})$ are $\mathbf{x}^{302}, \mathbf{x}^{311}, \mathbf{x}^{320}, \mathbf{x}^{312}, \mathbf{x}^{321}$ and their sum is the glide polynomial $\glide^{(1)}_{302}$ corresponding to $C$.

Hence the elements of $\overline{{\rm GSS}(a,b)}$ are partitioned by the elements of ${\rm GSS}(a,b)$, with the sum of the monomials in each part equal to the appropriate glide polynomial.
\end{proof}

We can use Theorem~\ref{thm:structure_constants_for_glide polynomials} to better understand ${\rm GSS}(a,b)$ and the glide polynomials appearing in the product $\glide_a\glide_b$.
For a weak composition $a$, let $z(a)$ denote the number of zeros in $a$ that precede a nonzero entry.

\begin{corollary}\label{cor:GSS_max_length}
If $\glide_c$ appears in the glide expansion of $\glide_a \glide_b$, then 
\[|c| \le |a| + |b| + z(a)+z(b).\]
Moreover, if $\glide_a$ and $\glide_b$ use the same number of variables, then this bound is attained by some glide polynomial $\glide_d$ in the glide expansion of $\glide_a \glide_b$.
\end{corollary}
\begin{proof}
By Theorem~\ref{thm:structure_constants_for_glide polynomials}, the length of an element of ${\rm GSS}(a,b)$ is the degree of the lowest-degree monomial of the corresponding glide polynomial. This degree is bounded above by the maximum possible degree of a monomial appearing in the product $\glide_a\glide_b$, i.e., the sum of the highest degrees of monomials in $\glide_a$ and $\glide_b$. These highest-degree monomials arise from glides of $a$ and of $b$ with as many red entries as possible. Since the number of red entries in a glide of $a$ is clearly at most $z(a)$, the greatest possible degree of a glide of $a$ is $|a| + z(a)$. The analogous statement holds for glides of $b$.

To see the bound is attained, first note that if $\glide_a$ and $\glide_b$ use the same number of variables then we may suppose that neither $a$ nor $b$ have trailing zeros (by deleting trailing zeros of $a$ and $b$ if necessary). Suppose we have a glide $a'$ of $a$ and a glide $b'$ of $b$, each with as many red entries as possible. Then both $a'$ and $b'$ must have no zero entries at all. Let $D \in \overline{{\rm GSS}(a,b)}$ be the image of $(a', b')$ under the map from $M(a,b)$ to $\overline{{\rm GSS}(a,b)}$ given in the proof of Theorem~\ref{thm:structure_constants_for_glide polynomials}. We claim that in fact, $D \in {\rm GSS}(a,b)$. Suppose for a contradiction that $D$ has two adjacent copies of the same letter; without loss of generality, we have $b'_i$ is black and $b'_{i+1}$ is red, the letters of $A^\gen$ in the $i$th run of $D$ are smaller than the letters of $B^\gen$ in this run, and the letters of $A^\gen$ in the $(i+1)$th run of $D$ are larger than the letters of $B^\gen$ in the $(i+1)$th run. But this is impossible since $a'_i$ and $a'_{i+1}$ are both nonzero, and letters of $A^\gen$ decrease from right to left. Therefore $D$ does not have two adjacent copies of the same letter. Moreover, $D$ cannot have a bar between an ascent, since clearly that would require some $a'_i$ or $b'_i$ to be zero.
Thus $D \in {\rm GSS}(a,b)$, and so $\glide_d$ appears in the product $\glide_a\glide_b$, where $d = {\sf Comp}({\sf BumpRuns}(D))$.
\end{proof}


\begin{corollary}\label{corollary:interval}
If $\glide_c$ appears in the glide expansion of $\glide_a \glide_b$ and $|c| > |a| + |b|$, then there is a glide polynomial $\glide_d$ in the glide expansion of $\glide_a \glide_b$ with $|d| = |c| - 1$.
\end{corollary}
\begin{proof}
By Theorem~\ref{thm:structure_constants_for_glide polynomials}, there is a $C \in {\rm GSS}(a,b)$ corresponding to the weak komposition $c$. Since $|c| > |a| + |b|$, $C$ has at least one letter appearing more than once. Let $D$ be the subword formed from $C$ by deleting the rightmost letter of $C$ that is a repeat. Note that $D \in A \shuffle_\gen B$. Since the set of genotypes of $D$ is a \emph{subset} of the set of genotypes of $C$ and $C \in {\rm GSS}(a,b)$,  all genotypes of $D$ satisfy the dominance conditions. Thus $D \in {\rm GSS}(a,b)$. The corollary follows by taking $d = {\sf Comp}({\sf BumpRuns}(D))$.
\end{proof}



Theorem~\ref{thm:structure_constants_for_glide polynomials} and the positive combinatorial expansion of a Grothendieck polynomial in the glide basis (Theorem~\ref{thm:glide_expansion}) together yield a positive Littlewood-Richardson rule for the expansion of a product of Grothendieck polynomials in the glide basis. For a permutation $w$, let $\inv(w)$ denote the number of inversions of $w$.

\begin{theorem}
  For a weak composition $a$ and permutations $u$ and $v$, we have
\[    \cgroth_u \cgroth_v = \sum_{a} \beta^{|a|-\inv(u)-\inv(v)} c_{u,v}^{a} \glide_{a},\]
  where
\[    c_{u,v}^{a} = \!\!\!\!\!\!\!\!\!\! \sum_{(P,Q) \in \QPD(u) \times \QPD(v)} \!\!\!\!\! \!\!\!\!\! g_{\wt(P), \wt(Q)}^a.\]
\end{theorem}

\begin{proof}
Immediate from Theorems~\ref{thm:glide_expansion} and \ref{thm:structure_constants_for_glide polynomials}.
\end{proof}


%

\section*{Acknowledgements}
We thank Cara Monical for sharing helpful code for computing with Grothendieck polynomials, and Travis Scrimshaw for programming support. We thank Becky Patrias for useful conversation about \cite{Lam.Pylyavskyy}.

OP was partially supported by an NSF Graduate Research Fellowship.

%
%

\bibliographystyle{amsalpha} 
\bibliography{grothenslide.bib}

\end{document}